\documentclass[12pt]{amsart}
\usepackage{amsmath, amscd, amsthm, amssymb, graphics, mathrsfs, setspace,fancyhdr,geometry,tabularx,shapepar, hyperref}
\usepackage[all]{xy}
\hypersetup{colorlinks=false}
\theoremstyle{plain}
\newtheorem{thm}{Theorem}[section]

\newtheorem{lem}[thm]{Lemma}

\newtheorem{prop}[thm]{Proposition}

\newtheorem*{thm*}{Theorem}
\newtheorem*{problem*}{Problem}

\theoremstyle{definition}
\newtheorem{defn}[thm]{Definition}

\theoremstyle{remark}
\newtheorem{rem}[thm]{Remark}

\numberwithin{equation}{section}

\newcommand{\Spec}{\operatorname{Spec}}

\newcommand{\pone}{{\mathbb P}^1}

\newcommand{\SB}{\operatorname{SB}}

\newcommand{\id}{\operatorname{id}}

\newcommand{\gr}{\operatorname{Gr}}

\begin{document}

\title[$K_1$-Zero-Cycles for Generalized Sever-Brauer Varieties]{The Group of $K_1$-Zero-Cycles on the Second Generalized Severi-Brauer Variety of an Algebra of Index 4}

\author{Patrick K. McFaddin}
\address{Department of Mathematics, University of South Carolina, 
Columbia, SC 29208}
\email{pkmcfaddin@gmail.com}
\urladdr{\url{http://mcfaddin.github.io/}}
\thanks{Partially supported by NSF Grant DMS-1151252}

\subjclass[2010]{Primary \href{http://www.ams.org/msc/msc2010.html?t=16K20&s=14c25&btn=Search&ls=Ct}{14C25}, Secondary \href{http://www.ams.org/msc/msc2010.html?t=&s=16k20&btn=Search&ls=s}{16K20}}

\maketitle
\addtocounter{section}{0}

\begin{abstract}
In this manuscript, it is shown that the group of $K_1$-zero-cycles on the second generalized Severi-Brauer variety of an algebra $A$ of index 4 is given by elements of the group $K_1(A)$ together with a square-root of their reduced norm.  Utilizing results of Krashen concerning exceptional isomorphisms, we translate our problem to the computation of cycles on involution varieties.  Work of Chernousov and Merkurjev then gives a means of describing such cycles in terms of Clifford and spin groups and corresponding $R$-equivalence classes.  We complete our computation by giving an explicit description of these algebraic groups.
\end{abstract}

\section{Introduction}

The theory of algebraic cycles on homogenous varieties has seen many useful applications to the study of central simple algebras, quadratic forms, and Galois cohomology.  Significant results include the Merkurjev-Suslin Theorem \cite{MS82} and the Milnor and Bloch-Kato Conjectures (now theorems of Orlov-Vishik-Voevodksy and Voevodsky-Rost-Weibel) \cite{Voe}.  Despite these successes, a general description of Chow groups (with coefficients) remains elusive, and computations of these groups are done in various cases, e.g., \cite{ChernMerkSpin, ChernMerk, KrashZeroCycles, Mer95, MerkRat, MerkSus, PSZ, Swan}.

The impetus for the investigation presented here is a result of Merkurjev and Suslin concerning cycles on Severi-Brauer varieties- twisted forms of projective space.  For a variety $X$ over a field $F$, there is a complex \cite{Quillen, Rost}$$\cdots \to \coprod _{x \in X_{(p+1)}} K_{p+q+1}(F(x)) \to \coprod _{x \in X_{(p)}} K_{p+q}(F(x)) \to \coprod_{x \in X_{(p-1)}} K_{p+q-1} (F(x)) \to \cdots $$ where $X_{(i)}$ is the set of points of $X$ of dimension $i$, $K_i$ are Quillen $K$-groups, and the differentials are given component-wise by residue homomorphisms.  The $p^{\text{th}}$ homology group of this complex, i.e., homology at the middle term above, is denoted $A_p(X, K_q)$.  Taking $X$ to be the Severi-Brauer variety associated to a central simple algebra $A$, it is shown in \cite{MerkSus} that the group $A_0(X, K_1)$ of $K_1$-zero-cycles is given by $K_1(A)$.  

Here we provide a similar result for cycles on generalized Severi-Brauer varieties, i.e., twisted forms of Grassmannians.  We show that the group of $K_1$-zero-cycles on the second generalized Severi-Brauer variety of an algebra $A$ of index 4 is given by elements of $K_1(A)$ together with a square root of their reduced norm.

\begin{thm*}$\operatorname{(\ref{thm:mainthm})}$
Let $A$ be a central simple algebra of index 4 and arbitrary degree over a field $F$ of characteristic not 2, and let $X$ be the second generalized Severi-Brauer variety of $A$.  There is a group isomorphism $$A_0(X, K_1) \simeq  \{ (x, \alpha) \in K_1(A) \times F^{\times} \mid \operatorname{Nrd}_A(x) = \alpha ^2\}.$$
\end{thm*}

The paper is organized as follows.  In Section 2, we recall some definitions and facts about central simple algebras, involutions, and Severi-Brauer and involution varieties.  We then define cycle modules and $K$-cohomology.  

In Section 3, we introduce certain algebraic groups arising from algebras with involution. These include various unitary groups, the Clifford group, and spin group.  We also define the notion of $R$-equivalence and its use in relating these groups to $K$-theory and $K$-cohomology, using results of Voskresenski\u{i} and Chernousov-Merkurjev.

In Section 4, we prove Lemma~\ref{lem:red}, which gives a reduction to algebras of square degree.  Given an algebra $A$ of index 4, we reduce our computation of the group of $K_1$-zero-cycles on the second generalized Severi-Brauer variety of $A$ to that of the second generalized Severi-Brauer variety of $D$, where $A = M_n(D)$.  The fact that $D$ has degree 4 allows one to utilize a result of Krashen relating this variety to a certain involution variety.

In Section 5, we prove our main result.  After transferring our computation into the realm of involution varieties, we utilize work of Chernousov and Merkurjev expressing the group of $K_1$-zero-cycles in terms of the Clifford and spin groups of a certain algebra with involution.  An analysis of these groups completes the proof.

\subsection*{Acknowledgments}
I would like to express my sincere gratitude to Danny Krashen for posing this problem to me and for all of the helpful suggestions he has provided.  I would also like to thank the anonymous referee for useful remarks and suggestions.

\section{Preliminaries and Notation}

Throughout, $F$ denotes an infinite perfect field of characteristic unequal to 2.  While this assumption on the characteristic is only necessary for the proofs of Proposition~\ref{prop:deg4} and Theorem~\ref{thm:mainthm}, it aids in adding clarity to the exposition, eliminating the need for the use of quadratic pairs.  If $L/F$ is a field extension and $A$ is an $F$-algebra, we write $A_L:= A \otimes _F L$.  By a \emph{scheme} we will mean a separated scheme of finite type over the field $F$.  A variety is an integral scheme.  If $X$ and $Y$ are $F$-varieties, we write $X_Y : = X\times _{\Spec F} Y$.

A \emph{central simple algebra} over $F$ is a finite-dimensional $F$-algebra with no two-sided ideals other than $(0)$ and $(1)$ and whose center is precisely $F$.  Unless otherwise stated, all algebras will be central simple over $F$.  Recall that the dimension of a central simple algebra $A$ is a square, and we define the \emph{degree} of $A$ to be $\text{deg}(A) = \sqrt{\dim A}$.  One may write $A = M_n(D)$ for a division algebra $D$, unique up to isomorphism, and we define the \emph{index} of $A$ to be $\text{ind}(A) = \text{deg}(D).$  We say two central simple algebras over $F$ are \emph{Brauer-equivalent} if their underlying division algebras are isomorphic.

By an \emph{algebra with involution} we will mean a pair $(A,\sigma)$ where $A$ is a central simple algebra and $\sigma: A \to A$ is an anti-automorphism satisfying $\sigma ^2 = \id _A$.  An involution \emph{of the first kind} satisfies $\sigma |_F = \id _F$, while an involution \emph{of the second kind} induces a nontrivial degree 2 automorphism of $F$.  We refer to involutions of the second kind as \emph{unitary involutions}.  An involution of the the first kind which is a twisted form of a symmetric bilinear form is \emph{orthogonal}.  Otherwise, it is \emph{symplectic}.  The assumption that $\text{char}(F) \neq 2$ implies that algebras with quadratic pair \cite[\S 5.B]{MerkBook} are simply algebras with orthogonal involution.

It will be useful to extend the notion of ``unitary involution" to include semi-simple $F$-algebras of the form $A_1 \times A_2$, where each $A_i$ is a central simple over $F$.  The center $L$ of an algebra with unitary involution $(A, \sigma)$ will generally be an \'{e}tale quadratic extension of $F$, i.e., either $L \simeq F \times F$ or $L/F$ is a separable quadratic field extension. In the first case, $A \simeq A_1 \times A_2$ as above, and in the second case $A$ is a central simple algebra over $L$.  We will refer to such an algebra as a \emph{central simple algebra with unitary involution}, even though the algebra is not necessarily simple and its center is not $F$ (see introduction to \cite[\S 2.B]{MerkBook}). 

\subsection{Severi-Brauer and Involution Varieties}

The following may be found in \cite{Blanch, MerkBook, KrashZeroCycles, Tao}.  Let $A$ be a central simple algebra of degree $n$.  For any integer $1 \leq k \leq n$, the $k^{\text{th}}$ generalized Severi-Brauer variety $\SB_k(A)$ of $A$ is the variety of right ideals of dimension $nk$ in $A$.   Such a varietiy is a twisted form of the Grassmannian $\gr(k, n)$ of $k$-dimensional subspaces of an $n$-dimensional vector space.  The variety $\SB_1(A) = \SB(A)$ is the usual Severi-Brauer variety of $A$, which is a twisted form of projective space.  For a field extension $L/F$, the variety $\SB_k(A)$ has an $L$-rational point if and only if $\text{ind}(A_L) \mid k$ \cite[Prop. 1.17]{MerkBook}.  We concern ourselves with the case $k = 2$, the second generalized Severi-Brauer variety of $A$.

Let $(A, \sigma)$ be an algebra with orthogonal involution and let $I$ be a right ideal of $A$.  The \emph{orthogonal ideal} of $I$ with respect to $\sigma$ is given by $$I^{\perp} = \{ x \in A \mid \sigma(x)y = 0 \text{ for } y \in I\}.$$  We say that an ideal $I$ is \emph{isotropic} if $I \subset I ^{\perp}$.  Let $\text{IV}(A,\sigma)$ denote the collection of isotropic ideals in $A$ of dimension $n$.  It is a subvariety of $\SB(A)$, with inclusion defined by forgetting the isotropy condition.  Just as $\SB(A)$ is a twisted form of projective space, $\text{IV}(A, \sigma)$ is a twisted form of a projective quadric.

\subsection{Cycle Modules}
Cycle modules were first introduced by M. Rost in \cite{Rost} and good references are \cite{GMS} and \cite{EKM} for the case of Milnor $K$-theory.  

\begin{defn} 
A \emph{cycle module} $M$ over $F$ is a function assigning to every field extension $L/F$ a graded abelian group $M(L) = M_*(L)$, which is a graded module over the Milnor $K$-theory ring $K_*^M(F)$ satisfying some data and compatibility axioms.  This data includes
\begin{enumerate}
\item For each field homomorphism $L \to E$ over $F$, there is a degree 0 homomorphism $r_{E/L}: M(L) \to M(E)$ called \emph{restriction}.
\item For each field homomorphism $L \to E$ over $F$, there is a degree 0 homomorphism $c_{E/L}: M(E) \to M(L)$ called \emph{corestriction} (or norm).
\item For each extension $L/F$ and each (rank 1) discrete valuation $v$ on $L$, there is a degree $-1$ homomorphism $\partial _{v}: M(L) \to M(\kappa(v))$ called the \emph{residue homomorphism}, where $\kappa(v)$ is the residue field of $v$.
\end{enumerate}
These homomorphisms are compatible with the corresponding maps in Milnor $K$-theory. See D1-D4, R1a-R3e, FD, and C of \cite[Def. 1.1, Def. 2.1]{Rost}.
\end{defn}

Let $X$ be an $F$-variety and $M$ a cycle module over $F$.  Utilizing this axiomatic framework, we set $$C_p(X, M_q) = \coprod _{x \in X_{(p)}} M_{p+q}(F(x)),$$ where $X_{(p)}$ denotes the collection of points of $X$ of dimension $p$, to obtain a complex $$ \cdots \xrightarrow{} C_{p+1}(X, M_{q}) \xrightarrow{d_X} C_p(X, M_q) \xrightarrow{d_X} C_{p-1}(X, M_{q}) \xrightarrow{} \cdots$$  with differentials $d_X$ induced by the homomorphisms $\partial _v$ associated to valuations of codimension 1 subvarieties.  This is often referred to as the \emph{Rost complex} \cite{EKM} or \emph{Gersten complex} \cite{Panin}.  We denote the homology group at the middle term by $A_p(X, M_q)$.  Our main focus will be the group $$A_0(X, K_1) = \operatorname{coker} \left( \coprod _{x\in X_{(1)}} K_2(F(x))\xrightarrow{d_X} \coprod _{x \in X_{(0)} } K_1(F(x)) \right)$$ of $K_1$-zero-cycles, where $K_*$ is the cycle module given by Quillen $K$-theory.  More concretely, we wish to describe the collection of equivalences classes of formal sums $\sum (\alpha, x)$, where $x$ is a closed point on $X$ and $\alpha \in F(x)^{\times}$.

\section{Involutions, Algebraic Groups, and $R$-Equivalence}  We define certain algebraic groups associated to central simple algebras with involution.  These groups are twisted forms of their more classical counterparts.  Their importance comes from the fact that algebras with involution of type $A_3$ are closely related to Clifford and spin groups associated to algebras with involution of type $D_3$, coming from the exceptional identification of the $A_3$ and $D_3$ Dynkin diagrams (this is made precise in Proposition~\ref{prop:excid}).  Furthermore, utilizing a result of Chernousov and Merkurjev stated below, these groups can be used to compute $K_1$-zero-cycles for involution varieties associated to algebras with involution which have index no more than 2.

An \emph{algebraic group} is a smooth affine group scheme.  We follow the same notational convention as \cite{ChernMerkSpin}, utilizing $\text{GL}_1(A)$, $\text{SL}_1(A)$, $\text{Spin}(A, \sigma)$, etc., to denote the groups of $F$-points of the corresponding algebraic groups $\textbf{GL}_1(A),$ $\textbf{SL}_1(A)$, $\textbf{Spin}(A, \sigma)$, etc.  We suppress reference to the scheme structure and focus only on the collections of $F$-points, although we continue to refer to them as ``algebraic groups."

\subsection{Involutions and Algebraic Groups}  A good reference for the following material is \cite{MerkBook}.  Let $A$ be a central simple algebra over $F$, and consider the algebraic group $\text{GL}_1(A) = A^{\times}$ of invertible elements in $A$, called the \emph{general linear group} of $A$.  The kernel of the reduced norm homomorphism $\text{Nrd}_A: \text{GL}_1(A) \to F^{\times}$ \cite[$\S$1.6]{MerkBook} is denoted $\text{SL}_1(A)$, called the \emph{special linear group} of $A$.

Let $(A, \sigma)$ be a central simple algebra with involution.  A \emph{similitude} of $(A, \sigma)$ is an element $g \in A$ satisfying $\sigma(g) g \in F^{\times}$.  The collection of all similitudes of $(A, \sigma)$ is denoted $\text{Sim}(A, \sigma)$.  The scalar $\mu(g) := \sigma(g) g$ is called the \emph{multiplier} of $g$.  The association $g \mapsto \mu(g)$ defines a group homomorphism $\mu: \text{Sim}(A, \sigma) \to F^{\times}$.  If the involution $\sigma$ is of unitary type, we denote $\text{Sim}(A, \sigma)$ by $\text{GU}(A, \sigma)$ and call it the \emph{general unitary group} of $(A, \sigma)$.

 Let $(A, \sigma)$ be an algebra with unitary involution, with $\text{deg}(A) = 2m$. Define the \emph{special general unitary} and \emph{special unitary} groups  $$\text{SGU}(A, \sigma) =\{g \in \text{GU}(A, \sigma) \mid \text{Nrd}_A(g) = \mu(g) ^ m\}$$ $$ \text{SU}(A, \sigma) = \{ u \in \text{GU}(A, \sigma)\mid \text{Nrd}_A(u) = 1\}.$$

In the case where $A$ has center $L \simeq F \times F$, there is a central simple algebra $E$ over $F$ such that $(A, \sigma) \simeq (E \times E^{\text{op}}, \varepsilon)$, where $\varepsilon$ is the involution which switches factors (the \emph{exchange} involution) \cite[Prop. 2.14]{MerkBook}.  The general, special general, and special unitary groups of $(A, \sigma)$ are then given by \cite[\S 14.2]{MerkBook} $$\text{GU}(E \times E^{\text{op}}, \varepsilon) = \{(x, \alpha(x^{-1})^{\text{op}}) \mid \alpha \in F^{\times}, x \in E^{\times}\} \simeq E^{\times} \times F^{\times}$$ $$\text{SGU}(E \times E^{\text{op}}, \varepsilon) = \{(x, \alpha) \in E^{\times} \times F^{\times} \mid \text{Nrd}_E(x) = \alpha ^m\}$$ $$\text{SU}(E\times E^{\text{op}}, \varepsilon) = \{x \in E^{\times} \mid \text{Nrd}_E (x) = 1\} = \text{SL}(E).$$

\subsection{Clifford and Spin Groups}  
Given an algebra with orthogonal involution $(A, \sigma)$, the \emph{Clifford algebra} $C(A, \sigma)$ is an $F$-algebra which is a quotient of the tensor algebra of $A$.  Its multiplication is defined in terms of the involution $\sigma$ and it is a twisted form of the even Clifford algebra associated to a quadratic space.  Together with its canonical involution, the Clifford algebra enjoys the structure of a central simple algebra with unitary, orthogonal, or symplectic involution, depending on its degree and the characteristic of $F$ \cite[Prop. 8.12]{MerkBook}. We will be interested in cases where this canonical involution is unitary, so the center of the Clifford algebra is an \'{e}tale quadratic extension of $F$.  The multiplicative group of $C(A, \sigma)$ contains a group $\Gamma(A, \sigma)$, called the \emph{Clifford group}, whose action on $C(A, \sigma)$ fixes $A \subset C(A, \sigma)$.  One defines a certain \emph{multiplier map} $\underline{\mu}: \Gamma(A, \sigma) \to F^{\times}$ whose kernel is called the \emph{spin group} of $(A, \sigma)$, denoted $\text{Spin}(A, \sigma)$. We refer the reader to \cite[$\S\S$ 8, 13]{MerkBook} for all pertinent definitions.

If $(A, \sigma)$ is an algebra with unitary involution of degree $n = 2m$, the \emph{discriminant algebra} of $(A, \sigma)$, denoted $D(A, \sigma)$, is a central simple algebra of degree ${ n \choose m}$ which plays the role of the the exterior algebra.  If $F^s$ denotes a separable closure of $F$, then $A_{F^s} \simeq \text{End}_{F^s} (V)$ and $D(A, \sigma)_{F^s} = \text{End}_{F^s}(\wedge ^m V)$.  The discriminant algebra comes equipped with a so-called canonical orthogonal involution $\underline{\sigma}$ induced by the involution $\sigma$ (see \cite[$\S$10.E]{MerkBook}).

Let us make precise the relationship between the Clifford, spin, and unitary groups arising from the exceptional identifications of algebras of  types $A_3$ and $D_3$.  We summarize the discussion found in \cite[$\S$15.D]{MerkBook}.

 Let $\mathsf{A}_3$ denote the groupoid of central simple algebras with unitary involution of degree 4, $\mathsf{D}_3$ the groupoid of central simple algebras $F$-algebras with orthogonal involution of degree 6.  There are functors $\delta: \mathsf{A}_3 \to \mathsf{D}_3$ and $\gamma: \mathsf{D}_3 \to \mathsf{A}_3$, where $\gamma$ defined by taking an algebra with orthogonal involution to its Clifford algebra with canonical involution, and $\delta$ is defined by taking an algebra with unitary involution to its discriminant algebra with canonical orthogonal involution.  These functors define an equivalence of groupoids \cite[Theorem 15.24]{MerkBook}.
 
 \begin{prop} [\cite{MerkBook}, Prop.15.27]\label{prop:excid}
 Let $(A, \sigma)$ and $(B, \tau)$ correspond to one another under the groupoid equivalence of $\mathsf{A}_3$ and $\mathsf{D}_3$.  Then we have identifications $\Gamma(A, \sigma) = \operatorname{SGU}(B, \tau)$ and $\operatorname{Spin}(A, \sigma) = \operatorname{SU}(B, \tau).$
 \end{prop}
 
\subsection{$R$-Equivalence} 
See \cite{ChernMerkSpin}. Let $G$ be an algebraic group over $F$.  A point $x \in G(F)$ is called $R$-\emph{trivial} if there is a rational morphism $f: \pone \dashrightarrow G$, defined at 0 and 1,  and with $f(0) = 1$ and $f(1) = x$.  The collection of all $R$-trivial elements of $G(F)$ is denoted $RG(F)$ and is a normal subgroup of $G(F)$.  If $H$ is a normal closed subgroup of $G$ then $RH(F)$ is a normal subgroup of $G(F)$ \cite[Lemma 1.2]{ChernMerkSU}.

The notion of $R$-equivalence is related to algebraic $K$-groups and $K$-cohomology groups, as we see in the following results.\\

\noindent \textbf{Example.}\label{K1Ex} See \cite{Vos}. Given a central simple algebra $A$, the abelianization map $ \text{GL}_1(A) = A^{\times} \to A^{\times}_{\text{ab}} = K_1(A)$ induces an isomorphism $$\text{GL}_1(A)/R\text{SL}_1(A) \simeq K_1(A).$$

\begin{thm}[\cite{ChernMerkSpin}]\label{thm:ChernMerkSpin}
Let $A$ be a central simple algebra over $F$ of even dimension and index at most 2 with quadratic pair $(\sigma, f)$, $X$ be the corresponding involution variety.  Then there is a canonical isomorphism $$\Gamma(A, \sigma, f) / R \operatorname{Spin}(A, \sigma, f) \simeq A_0(X, K_1).$$
\end{thm}

\section{Reduction to Algebras of Square Degree}  
Given a prime integer $p$ and an algebra $A$ of index $p^2$, we reduce the computation of $K_1$-zero-cycles of $\SB_p(A)$ to that of $\SB_p(D)$, where $D$ is the underlying division algebra of $A$.  For $p =2$, this reduction to algebras of degree 4 will allow the use of involution varieties in the proof of the main theorem.

For $J$ a right ideal of $A$, define $\SB_k(J)$ as the collection of right ideals of $A$ of reduced dimension $k$ which are contained in $J$ \cite[Def. 4.7]{KrashZeroCycles}.

\begin{lem}\label{lem:red}
Let $p$ be a prime integer and let $A = M_n(D)$ be a central simple algebra of index $p^2$.  Let $X = \SB_p(A)$, and $Y =   \SB_p(D)$.  There is an isomorphism $A_0(X, K_1) \simeq A_0(Y, K_1)$.
\end{lem}

\begin{proof}
Fix an ideal $J$ of $ A$ of reduced dimension $p^2$, the existence of which follows from $\text{ind}(A) = p^2$, and let $e_J$ be the corresponding idempotent element of $A$ \cite[\S 12]{MerkBook}.  Define a rational map $\varphi _J: \SB_p(A) \dashrightarrow \SB_p(J)$ by the association $I \mapsto e_J I \subset J$, for any ideal $I $ of reduced dimension $p$.  The map $\varphi _J$ is defined on the open locus consisting of ideals $I$ satisfying $\text{rdim}(e_J I) = p.$  

By \cite[Thm. 4.8]{KrashZeroCycles}, the algebra $D:= e_J A e_J$ is degree $p^2$, Brauer-equivalent to $A$ and satisfies $\SB_p(J) = \SB_p(D)$.  Since $\text{ind}(A)= p^2$, the algebra $D$ is division and $A = M_n(D)$.  We denote the resulting map $\SB_p(A) \dashrightarrow \SB_p(D)$ also by $\varphi _J$.

Let $\eta$ be the generic point of $\SB_p(D)$ and take $L = F(\eta)$.  Let $\mathfrak{f}$ be the generic fiber of $\varphi _J$, i.e., $\mathfrak{f} = \SB_p(A)_{L} = \SB_p(A_{L})$ is the scheme-theoretic fiber over $\eta$. We first show that $\mathfrak{f}$ is a rational $L$-variety.  The field $L$ satisfies $\text{ind} (D_{L}) \mid p$.  Since $D$ is Brauer-equivalent to $A$, we have $\text{ind}(A_{L}) \mid p$, so that $\SB_p(A)$ has an $L$-rational point.  By \cite[Prop. 3]{Blanch}, the function field of $\SB_p(A)_{L}$ is purely transcendental over $L$, so that $\mathfrak{f}= \SB_p(A)_{L}$ is rational.  Therefore, $\SB_p(D)_{\mathfrak{f}}$ is birational to $\SB_p(D) \times \mathbb{P}_L ^{\dim \mathfrak{f}}$.  

The group of $K_1$-zero-cycles is an invariant of smooth projective varieties \cite[Cor. RC.13]{KM}, so that we have an isomorphism $$A_0(\SB_p(D)_{\mathfrak{f}}, K_1) \simeq A_0(\SB_p(D) \times \mathbb{P}_L ^{\dim \mathfrak{f}}, K_1).$$ This isomorphism in conjunction with the Projective Bundle Theorem \cite[Theorem 53.10]{EKM}, yields $$A_0(\SB_p(D) _{\mathfrak{f}}, K_1) \simeq A_0(\SB_p(D), K_1). $$ The variety $\SB_p(A)$ is birational to $\SB_p(D)_{ \mathfrak{f}}$, isomorphic along the open locus of definition of $\varphi _J$.  Again using that $A_0(-, K_1)$ is a birational invariant, and combining this with the above isomorphisms gives $A_0(\SB_p(A), K_1) \simeq A_0(\SB_p(D), K_1).$

\end{proof}

\section{Main Result}

Having reduced our computation to the case of an algebra of degree 4, we utilize a result of Krashen \cite{KrashZeroCycles} to transfer the computation of zero-cycles of the second generalized Severi-Brauer variety to that of an involution variety.  This result also guarantees that the involution variety of interest comes from an algebra of index no greater than 2.  We then make use of Theorem~\ref{thm:ChernMerkSpin} to translate this computation into an analysis of those algebraic groups defined in $\S 3$.

\begin{prop}\label{prop:deg4}
Let $A$ be a central simple algebra of degree 4 over a field $F$ and let $X$ be the second generalized Severi-Brauer variety of $A$.  There is a group isomorphism $$A_0(X, K_1) \simeq \{(x, \alpha) \in K_1(A) \times F^{\times} \mid \operatorname{Nrd}_A(x) = \alpha ^2 \}.$$
\end{prop}

\begin{proof}

By \cite[Lem. 6.5]{KrashZeroCycles}, $\SB_2(A)$ is isomorphic to the involution variety $\operatorname{IV}(B, \sigma)$ of a degree 6 algebra $B$ with orthogonal involution $\sigma$.  In particular, $$A_0(\text{SB}_2(A), K_1) = A_0(\text{IV}(B, \sigma), K_1).$$  Moreover, $B$ is Brauer-equivalent to $A^{\otimes 2}$, so that $\operatorname{ind }(B) \leq 2$.  The involution $\sigma$ is obtained from the bilinear form $$\wedge ^2 V \times \wedge ^2 V \to \wedge ^4 V \simeq F$$ by descent.  Consider the algebra $(A\times A^{\text{op}}, \varepsilon)$ over $F \times F$ with unitary involution defined by exchanging factors.  Its discriminant algebra $D(A \times A^{\text{op}}, \varepsilon)$ is given by $\lambda ^2 A$ \cite[10.31]{MerkBook}, which has degree 6 and is Brauer-equivalent to $A ^{\otimes 2}$ \cite[10.4]{MerkBook}.  The canonical involution $\underline{\varepsilon}:= \varepsilon^{\wedge 2}$ \cite[10.31]{MerkBook} on $\lambda ^2 A$ is also induced by the bilinear form $\wedge ^2 V \times \wedge ^2  V \to \wedge ^4 V \simeq F$  \cite[Def. 10.11]{MerkBook}, yielding an isomorphism $(B, \sigma) \simeq (D(A \times A^{\text{op}}, \varepsilon), \underline{\varepsilon})$ of algebras with orthogonal involution.  Thus, $(B, \sigma)$ and $(A\times A^{\text{op}}, \varepsilon)$ correspond to one another under the groupoid equivalence of $\mathsf{A}_3 $ and $\mathsf{D}_3$.

Since $\text{ind}(B) \leq 2$, Theorem~\ref{thm:ChernMerkSpin} yields a canonical isomorphism \begin{equation} A_0(\text{IV}(B, \sigma), K_1) \simeq \Gamma(B,\sigma)/R\operatorname{Spin}(B, \sigma).\end{equation}\label{eq:CliffSpin} 
As $(B, \sigma) \in \mathsf{D}_3$ corresponds to $(A \times A^{\text{op}}, \varepsilon) \in \mathsf{A}_3$, there are exceptional identifications $\Gamma(B, \sigma) = \text{SGU}(A \times A^{\text{op}}, \varepsilon)$ and $ \text{Spin}(B, \sigma) = \text{SU}(A \times A^{\text{op}}, \varepsilon)$, as noted in Proposition~\ref{prop:excid}.  Furthermore, the discussion in $\S$3.1 gives $$ \text{SGU}(A \times A^{\text{op}}, \varepsilon) = \{(x, \alpha) \in A^{\times} \times F^{\times} \mid \text{Nrd}_A(x) = \alpha ^2\}$$ $$\text{SU}(A \times A^{\text{op}}, \varepsilon)= \text{SL}_1(A)$$ with the inclusion of the latter given by inclusion into the first factor $x \mapsto (x, 1)$.  The quotient in ~\ref{eq:CliffSpin}.1 can then be identified as $$\{(x, \alpha) \in A ^{\times} \times F^{\times} \mid \text{Nrd}_A(x) = \alpha ^2\}/ R\text{SL}_1(A),$$ and therefore consists of elements $ x \in A^{\times}/R\text{SL}_1(A) = \text{GL}_1(A)/R\text{SL}_1(A) = K_1(A)$ together with a square-root $\alpha$ of $\text{Nrd}_A(x)$ in $F^{\times}$.
\end{proof}

\begin{thm}\label{thm:mainthm}
Let $A$ be a central simple algebra of index 4 over a field $F$, and let $X$ be the second generalized Severi-Brauer variety of $A$.  There is a group isomorphism $$A_0(X, K_1) \simeq \{(x, \alpha) \in K_1(A) \times F^{\times} \mid \operatorname{Nrd}_A(x) = \alpha ^2 \}.$$
\end{thm}

\begin{proof} 
The reduced norm respects the canonical isomorphism $K_1(A) = K_1(M_n(D)) = K_1(D)$.  Combining the isomorphisms of Lemma~\ref{lem:red} and Proposition~\ref{prop:deg4} yields the desired result.
\end{proof}

\begin{rem}  If the cohomological dimension of $F$ is at most 2, Theorem~\ref{thm:mainthm} yields an isomorphism $A_0(X, K_1) \simeq F^{\times}$.  Indeed, if the reduced norm is surjective, the pair $(x, a) \in K_1(A) \times F^{\times}$ is completely determined by $a$ since $x = \text{Nrd}_A ^{-1}(a^2)$.  The surjectivity of the reduced norm is guaranteed by $\text{cd}(F) \leq 2$  \cite[Thm. 24.8]{Suslin} (also see \cite[Ex. III.3.2.a]{Serre}).  Examples of such fields include imaginary number fields and function fields of complex surfaces.  We thank the referee for this comment.
\end{rem}

\end{document}